\documentclass[12pt]{article}
\usepackage{amsmath,amssymb,amsbsy,amsfonts,amsthm,latexsym,amsopn,amstext,
                                amsxtra,euscript,amscd}

\begin{document}

\newtheorem{theorem}{Theorem}
\newtheorem{lemma}[theorem]{Lemma}
\newtheorem{claim}[theorem]{Claim}
\newtheorem{cor}[theorem]{Corollary}
\newtheorem{prop}[theorem]{Proposition}
\newtheorem{definition}{Definition}
\newtheorem{question}[theorem]{Open Question}
\def\bE{{\mathbf E}}
\def\mand{\qquad\mbox{and}\qquad}
\def\scr{\scriptstyle}
\def\\{\cr}
\def\({\left(}
\def\){\right)}
\def\[{\left[}
\def\]{\right]}
\def\<{\langle}
\def\>{\rangle}
\def\fl#1{\left\lfloor#1\right\rfloor}
\def\rf#1{\left\lceil#1\right\rceil}
\def\cA{{\mathcal A}}
\def\cB{{\mathcal B}}
\def\cC{{\mathcal C}}
\def\cD{{\mathcal D}}
\def\cE{{\mathcal E}}
\def\cF{{\mathcal F}}
\def\cG{{\mathcal G}}
\def\cH{{\mathcal H}}
\def\cI{{\mathcal I}}
\def\cJ{{\mathcal J}}
\def\cK{{\mathcal K}}
\def\cL{{\mathcal L}}
\def\cM{{\mathcal M}}
\def\cN{{\mathcal N}}
\def\cO{{\mathcal O}}
\def\cP{{\mathcal P}}
\def\cQ{{\mathcal Q}}
\def\cR{{\mathcal R}}
\def\cS{{\mathcal S}}
\def\cT{{\mathcal T}}
\def\cU{{\mathcal U}}
\def\cV{{\mathcal V}}
\def\cW{{\mathcal W}}
\def\cX{{\mathcal X}}
\def\cY{{\mathcal Y}}
\def\cZ{{\mathcal Z}}
\def\N{{\mathbb N}}
\def\Z{{\mathbb Z}}

\def\eps{\varepsilon}
\def\mand{\qquad\text{and}\qquad}
\def\tS{S^*}
\def\tsigma{\widetilde \sigma}
\def\C{{\mathbb C}}
\def\F{{\mathbb F}}
\def\Fp{\F_p}
\def\E{{\mathcal E}}
\def\e{{\mathbf e}}
\def\ep{\e_p}
\def\emm{\e_m}
\def\O{{\cO}}
\def\x{{\bf x}}
\def\y{{\bf y}}

\def\flm#1{{\left\lfloor#1\right\rfloor}_m}
\def \lrs{linear recurrence sequence}
\def \va {{\mathbf a}}
\def \vb {{\mathbf b}}
\def \vc {{\mathbf c}}
\def \ve {{\mathbf e}}
\def \vs {{\mathbf s}}
\def \vx{{\mathbf x}}
\def \vy{{\mathbf y}}
\def \vr {{\mathbf r}}
\def \vv {{\mathbf v}}
\def \vu{{\mathbf u}}
\def \vw{{\mathbf w}}
\def \vz {{\mathbf z}}
\def\eq{{\mathbf{\,e}}_q}

\def \bP{{\mathbf P}}

\newcommand{\comm}[1]{\marginpar{%
\vskip-\baselineskip 
\raggedright\footnotesize
\itshape\hrule\smallskip#1\par\smallskip\hrule}}



\title{On the Distribution of the Subset Sum
 Pseudorandom Number Generator on Elliptic Curves}

\author{
{\sc Simon R. Blackburn} \\
{Department of Mathematics}\\
{Royal Holloway University of London}\\
{Egham, Surrey, TW20 0EX, UK} \\
{\tt  s.blackburn@rhul.ac.uk}\\
\and
{\sc Alina Ostafe} \\
{Department of Computing}\\
{Macquarie University}\\
{Sydney, NSW 2109,
Australia}\\
{\tt alina.ostafe@mq.edu.au}
\and
{\sc Igor E. Shparlinski}\\
           {Department of Computing}\\
{Macquarie University}\\
{ Sydney, NSW 2109,
Australia}\\
{\tt igor.shparlinski@mq.edu.au} }
\date{}

\maketitle


\paragraph{Keywords:}\quad Pseudorandom numbers, Subset sum problem, Knapsack, Exponential
sums

\paragraph{MSC 2010:}\quad Primary 11K45,  11T71;  Secondary 11G05, 11T23, 65C05, 94A60

\newpage

\begin{abstract} Given a prime $p$, an elliptic curve $\E/\F_p$
over the finite field $\F_p$ of $p$ elements and a
binary \lrs\ $\(u(n)\)_{n =1}^\infty$
of order~$r$, we study the distribution of  the sequence of points
$$
\sum_{j=0}^{r-1} u(n+j)P_j, \qquad n =1,\ldots, N,
$$
on average over all possible choices of $\F_p$-rational points
$P_1,\ldots, P_r$ on~$\E$.  For a sufficiently large  $N$ we improve
and generalise a previous result in this direction due
to E.~El~Mahassni.
\end{abstract}

\section{Introduction}

The \textit{knapsack generator} or \textit{subset sum generator} is a
pseudorandom number generator introduced by Rueppel and
Massey~\cite{RuMa} and studied
in~\cite{Ruep1}; see also~\cite[Section~6.3.2]{MOV}
and~\cite[Section~3.7.9]{Ruep2}. It is defined as follows.  For an
integer $m\ge 1$ we denote by $\Z_m$ the residue ring modulo $m$. Let
$\(u(n)\)_{n =1}^\infty$ be a \lrs\ of order $r$ over the field of two
elements $\F_2$, see~\cite[Chapter~8]{LN}. Given an $r$-dimensional
vector $\vz = (z_0,\ldots\,, z_{r-1}) \in \Z_m^r$ of weights, we
generate a sequence of pseudorandom elements of $\Z_m$ by
\begin{equation}
\label{eq:SSG Zm}
\sum_{j=0}^{r-1}  u(n+j) z_j, \qquad n = 1, 2, \ldots .
\end{equation}

For cryptographic applications, it is usually
recommended to use a \lrs\ of maximal
period $\tau = 2^r-1$ and also the modulus $m = 2^r$.
Although the
results of~\cite{vzGaSh1,KnMe} suggest that this generator should be used
with care, no major attack against it is known.
In~\cite{ConShp,vzGaSh2}
results on the joint uniform  distribution of several consecutive
elements of this generator   have been
obtained (on average over all $r$-dimensional vectors
 $\vz = (z_0,\ldots\,, z_{r-1}) \in \Z_m^r$).

El~Mahassni~\cite{ElMa}
has recently considered the {\it elliptic curve subset sum\/}
generator and obtained some uniformity of distribution
results for this generator.  More precisely, let $p$ be a prime and
let $\E$ be an elliptic curve over the finite field $\F_p$ of $p$
elements.  Following~\cite{ElMa}, given a vector $\bP = (P_0, \ldots,
P_{r-1}) \in \E(\F_p)^r$ of $r$ points from the group $\E(\F_p)$ of
$\F_p$-rational points on $\E$ (see~\cite{Silv} for a background on
elliptic curves), we define the sequence:
\begin{equation}
\label{eq:SSG EC}
V_{\bP}(n)= \sum_{j=0}^{r-1}  u(n+j) P_j, \qquad n = 1, 2, \ldots,
\end{equation}
where the summation symbol refers to the group operation on $\E$; see
also~\cite{vzGaSh1}. If we fix any function $f: \E(\F_p) \to \F_p$, we
can define the output of the elliptic curve subset sum generator to be
the sequence $(f(V_{\bP}(n)))$. One of the simplest and most natural
choices for the function $f$ has been considered in~\cite{ElMa}, namely
$f(P) = x(P)$, the $x$-coordinate of any affine point
$P \in \E(\F_p)$. (We can define $x(O) = 0$ for the point at infinity $O$.)
With this choice for the function $f$, it is known~\cite{ElMa} that for almost all choices
of $\bP = (P_0, \ldots, P_{r-1}) \in \E(\F_p)^ s$, the sequence
$x\(V_{\bP}(n)\)/p$, $n=1,\ldots, N$, is uniformly distributed modulo
$1$ for a wide range of $N$.

In this paper we improve the result of~\cite{ElMa} on the distribution
of the sequence $x\(V_{\bP}(n)\)/p$, $n=1,\ldots, N$, in the case when
$N$ is sufficiently large, by adding some combinatorial arguments
to the existing techniques. We also establish results on the
distribution of the $s$-dimensional vectors
\begin{equation}
\label{eq:s-vec}
\(\frac{x\(V_{\bP}(n)\)}{p},\ldots, \frac{x\(V_{\bP}(n+s-1)\)}{p}\), \qquad n = 1,   \ldots, N,
\end{equation}
for any $s \ge 2$. (Note that we always assume that $\F_p$ is
represented by the set $\{0, \ldots, p-1\}$, so the
vectors~\eqref{eq:s-vec} belong to the $s$-dimensional unit cube.) The
methods in~\cite{ElMa} do not seem to extend to this case.  We note
that for small values of $N$ the results of~\cite{ElMa} remain the
only ones known for the elliptic curve subset sum generator. In
particular, full analogues of the results of~\cite{ConShp} are still
not known.

Throughout the paper, the implied constants in symbols `$O$' and
`$\ll$'  may depend on the integer parameter $s$. We recall that $U\ll V$
and $U = O(V)$ are both equivalent to the inequality $|U|\le cV$ with some
constant $c> 0$.

\section{Preliminaries}

\subsection{Discrepancy and Exponential Sums}

For a real $z$ and an integer $m \ge 1$ we use the notation
$$\e(z) = \exp(2\pi i z) \mand \emm(z) = \exp(2 \pi i z/m).
$$

For a sequence of $N$ points
\begin{equation}
\label{eq:Seq}
\Gamma = \(\gamma_{0,n}, \ldots\,,
\gamma_{s - 1,n} \)_{n=1}^N
\end{equation}
in the $s$-dimensional unit cube, we denote its {\it discrepancy\/} by
$D_\Gamma$.
That is,
$$D_\Gamma = \sup_{B \subseteq [0,1)^{s}}
\left|\frac{\cT_\Gamma(B)} {N} - |B|\right|,$$
where $\cT_\Gamma(B)$ is the number of points of the sequence $\Gamma$
in the box
$$B = [\alpha_0, \beta_0) \times \ldots \times [\alpha_{s - 1},
\beta_{s - 1})
\subseteq [0,1)^{s}$$
of volume $|B|$
and the supremum is taken over all such boxes.

As we have mentioned, one of our basic tools to study the uniformity of
distribution is the
Koksma--Sz\"usz inequality, which we present in a slightly weaker form
than that  given by Theorem~1.21
of~\cite{DrTi}.

For an integer vector $\va = (a_0, \ldots\,, a_{s-1}) \in \Z^{s}$ we define
$$
|\va| = \max_{\nu = 0, \ldots\,, s-1} |a_\nu|, \qquad
r(\va) = \prod_{\nu=0}^{s-1} \max\{|a_\nu|,\, 1\}.
$$

\begin{lemma}
\label{le:K-S} For any
integer $L > 1$ and any sequence $\Gamma$ of $N$
points~\eqref{eq:Seq} for the discrepancy $D_\Gamma$ we
have
$$D_\Gamma \ll  \frac{1}{L}
+ \frac{1}{N}\sum_{0 < |\va| < L } \frac{1}{r(\va)}
\left| \sum_{n=1}^N \e \(\sum_{\nu=0}^{s - 1}a_\nu\gamma_{\nu,n} \)\right|,
$$
where the sum is taken over all integer vectors
$\va = (a_0, \ldots\,, a_{s - 1}) \in \Z^{s}$
with $0 < |\va| < L$.
\end{lemma}

For estimation of  the corresponding exponential sums
with various  sequences of pseudorandom numbers,
the following special case of the bound of Bombieri~\cite{Bomb} is used.

\begin{lemma}
\label{lem:Bomb}
For  any rational function $f(X,Y) \in
\F_p(X,Y)$ of degree $d$ which
is not constant on an elliptic $\E$ over $\F_p$, the  bound
$$
\sum_{P \in \E(\F_p)}\hskip -25pt {\phantom{\sum}}^*\, \ep\(f(Q)\) \ll  d p^{1/2}
$$
holds,  where  $\sum{}^*$ means the the poles of $f(X,Y)$
are excluded from the summation.
\end{lemma}

We need the orthogonality relation:
\begin{equation}
\label{eq:ident}
\sum_{\eta =0}^{m-1} \emm( \eta \lambda) =
\left\{ \begin{array}{ll}
0,& \quad \mbox{if}\ \lambda\not \equiv 0 \pmod m, \\
m,& \quad \mbox{if}\  \lambda \equiv 0 \pmod m.
\end{array} \right.
\end{equation}
We also make use of the  inequality (which is immediate
from~\cite[Bound~(8.6)]{IwKow})
\begin{equation}
\label{eq:plogp}
\sum_{\eta =0}^{m-1}
\left| \sum_{\lambda=1}^M
\emm\(\eta \lambda\) \right|\ll m \log m,
\end{equation}
which holds for any integers $m$ and  $M$ with $1 \le M \le m$.

\subsection{Combinatorial Estimates}

Let $r$ and $s$ be positive integers such that $s\leq r$.
Write $\ve_1,\ldots,\ve_s$ for the standard orthogonal basis vectors of length $s$
and let $\mathbf{0}_s = (0,\ldots, 0)$ be the $s$-dimensional zero vector. We say
that  a
pair of  $r$-dimensional binary vectors $\vx=(x_0, \ldots, x_{r-1})$
and  $\vy=(y_0, \ldots, y_{r-1})$ is {\it $s$-good\/} if for all $h=1, \ldots, s$,
there exists at least one  pair $(i,j)$,
$0\le i,j\le r-s$  such that
$$
(x_i,x_{i+1},\ldots,x_{i+s-1})=\ve_h, \qquad (x_j,x_{j+1},\ldots,x_{j+s-1})=\mathbf{0}_s
$$
and
$$
(y_i,y_{i+1},\ldots,y_{i+s-1})=\mathbf{0}_s,  \qquad  (y_j,y_{j+1},\ldots,y_{j+s-1}) = \ve_h.
$$
We say that a pair $(\vx,\vy)$ is \emph{$s$-bad} if it is not $s$-good. We wish to obtain a bound on the number $f_s(r)$ of $s$-bad pairs of vectors of length $r$.

\begin{lemma}
\label{lem:Comp}
Let $s$ be a fixed positive integer. The number $f_s(r)$ of $s$-bad pairs of binary vectors of length $r$ is at most
$$
f_s(r)\le 2s 4^{s-1}\alpha_s ^r$$
where
$$
\alpha_s=\(4^{s}-1\)^{1/s}.
$$
\end{lemma}

\begin{proof}
We say that a pair $(\vx,\vy)$ is \emph{$(s,h)$-bad with respect to
  $\vx$} if there exists no integer $i$ with $0\leq i\leq r-s$ such
that $(x_i,x_{i+1},\ldots,x_{i+s-1})=\ve_h$ and
$(y_i,y_{i+1},\ldots,y_{i+s-1})=\mathbf{0}$.  Furthermore, we say that
a pair $(\vx,\vy)$ is \emph{$s$-bad with respect to $\vx$} if and only
if it is $(s,h)$-bad with respect to $\vx$ for some $h$.
Note that a pair $(\vx,\vy)$ is
$s$-bad if and only if for some $h$ the pair is $(s,h)$-bad with
respect to either $\vx$ or $\vy$.

Since there are at most $s$
possibilities for $h$, and the roles of $\vx$ and $\vy$ in the
definition of $s$-bad pairs are completely symmetrical, our bound
follows if we can prove that for any $s$ and $h$ the number of
$(s,h)$-bad pairs with respect to $\vx$ is at most
$4^{s-1}\alpha_s^r$.

Let $h$ be fixed. We bound the number of $(s,h)$-bad pairs $(\vx,\vy)$
with respect to $\vx$ as follows.  For an integer
$m=0,\ldots,\fl{r/s}-1$, there are at most $2^{2s}-1$ possibilities
for the pair
$$
((x_{ms}, x_{ms+1},\ldots,x_{ms+s-1}),(y_{ms}, y_{ms+1},\ldots,y_{ms+s-1}))
$$
of subsequences, since this pair of subsequences cannot be equal to $(\ve_h,\mathbf{0})$. So there are at most $(2^{2s}-1)^{\fl{r/s}}\leq(2^{2s}-1)^{r/s}=\alpha_s^r$ possibilities for the pair
$$
((x_0,x_1,\ldots,x_{\fl{r/s} s-1}),(y_0,y_1,\ldots,y_{\fl{r/s} s-1})).
$$ But $r-\fl{r/s} s\leq s-1$, and so there are at most $4^{s-1}$
possibilities for the last $r-\fl{r/s} s$ positions of $\vx$ and
$\vy$. This establishes our bound.
\end{proof}

In particular, since $\alpha_s<4$, we see from
Lemma~\ref{lem:Comp} that $f_s(r)=o(4^r)$ as $r\rightarrow\infty$
with $s$ fixed (and
so $s$-bad pairs are asymptotically rare).

We remark that it is not too difficult to see that $f_s(r)$ is bounded
below by $c_s\beta_s^r$ for some positive constants $c_s>0$ and
$\beta_s$ depending only on $s$. To see this we may use the
Perron--Frobenius Theorem, together with the fact that the number of
$(s,h)$-bad pairs with respect to $\vx$ is equal to the number of
walks of length $r-s$ in a certain directed graph (namely the tensor
product of two copies of a span $s$ binary de Bruijn graph, with a
single vertex removed). Indeed, for small values of $s$ computer
calculations based on this framework show that $f_s(r)\sim
c_s\beta_s^r$ where the value of $\beta_s$ is given in the following
table (to 5 decimal places), with the value of $\alpha_s$ given by our
upper bound included for comparison:
$$
\begin{array}{c|ccccc}
s&2&3&4&5&6\\\hline
\alpha_s&3.87298& 3.97906& 3.99609& 3.99922& 3.99984\\
\beta_s&3.73205&3.93947&3.98444&3.99615&3.99903
\end{array}
$$
The computer calculations show that the pairs that are $(s,h)$-bad
where $h=\lfloor (s-1)/2\rfloor$ and $h=\lceil (s-1)/2\rceil$ provide the dominant term for
$f_s(r)$ for $s\leq 6$.

\section{Main Result}

\subsection{One Dimensional Distribution}

For $\bP = (P_0, \ldots, P_{r-1}) \in \E(\F_p)^r$, we denote by $D_\bP(N)$ the discrepancy
of the points
$$
\(\frac{x(V_\bP(n))}{p}\), \qquad n =1,
\ldots, N.
$$

\begin{theorem}
\label{th:D} Let the \lrs\ $\(u(n)\)_{n =1}^\infty$ be
purely periodic with period $\tau$ and order $r = O(p^{1/2})$ and let its
characteristic polynomial be irreducible over $\F_2$.
Then for any $\delta >0$, and for all
except $O(\delta p^r)$ choices for $\bP \in \E(\F_p)^r$,
for all $1 \le N \le \tau$, we have
$$D_\bP(N) \ll \delta^{-1}\(N^{-1/2}+3^{r/2}N^{-1} p^{-1/4}+ p^{-1/2}\)(\log \tau)^2 \log p.$$
\end{theorem}

\begin{proof}
{From} Lemma~\ref{le:K-S}, used with $L = p$,  we derive
$$D_\bP(N) \ll \frac{1}{p} + \frac{1}{N}\sum_{0 < |a| < p}
\frac{1}{|a|}
\left|\sum_{n=1}^N
\e_p\(ax(V_\bP(n))\)\right|.$$
Let $N_\mu = \min\{2^{\mu}, \tau\}$, $\mu =0, 1, \ldots$. Define $k$ by the inequality
$N_{k-1} < N \le N_k$, that is, $k=\rf{\log_2N}$.
Then from~\eqref{eq:ident} we derive
\begin{eqnarray*}
\sum_{n=1}^N \e_p\(ax(V_\bP(n))\) = \frac{1}{N_k} \sum_{n=1}^{N_k}\sum_{\lambda=1}^N
\sum_{\eta=0}^{N_k}
\e_p\(ax(V_\bP(n))\) \e_{N_k}\(\eta(n-\lambda)\).
\end{eqnarray*}
Hence,
\begin{equation}
\label{DviaSigna}
D_\bP(N)  \ll \frac{1}{p} + \frac{1}{N N_k} \Delta_\bP(k)
\end{equation}
where
$$
\Delta_\bP(k) = \sum_{0 < |a| < p} \frac{1}{|a|}
\sum_{\eta=0}^{N_k} \left|\sum_{\lambda=1}^N \e_{N_k}\(-\eta\lambda\)\right| \left| \sum_{n=1}^{N_k}
\e_p\(ax(V_\bP(n))\) \e_{N_k}\(\eta n\)\right|.
$$
The celebrated Hasse bound shows that
\begin{equation}
\label{eq:Hasse}
\(\# \E(\F_p)\)^r \le (p^{1/2} + 1)^{2r} = O(p^r)
\end{equation}
as we have assumed that $r = O(p^{1/2})$. Applying the Cauchy inequality, we derive
\begin{eqnarray*}
\lefteqn{\( \sum_{\bP\in \E(\F_p)^r}
\left| \sum_{n=1}^{N_k}
\e_p\(ax(V_\bP(n))\)
\e_{N_k}\(\eta n\)\right|\)^2 }  \\
& & \qquad \le  (p^{1/2}+1)^{2r}\sum_{\bP\in \E(\F_p)^r} \left| \sum_{n=1}^{N_k}
\e_p\(ax(V_\bP(n))\) \e_{N_k}\(\eta n\)\right|^2 \\
& & \qquad \ll p^{r} \sum_{n,l=1}^{N_k} \e_{N_k}\(\eta (n-l)\)
\sum_{\bP\in \E(\F_p)^r}
\e_p\(a
\(x(V_\bP(n)) - x(V_\bP(l))\)\).
\end{eqnarray*}
For the case $n=l$, we estimate the inner sum trivially as
$\(\# \E(\F_p)\)^r = O(p^r)$.

We split the rest of the sum into two sums: the first over distinct $1$-bad pairs of vectors
\begin{equation}
\label{eq:pair}
\((u(n),\ldots, u(n+r-1)), (u(l)\ldots, u(l+r-1))\),
\end{equation} and the second over $1$-good
pairs of vectors~\eqref{eq:pair}.

Let $\cB_r$ be the set of pairs of indices $(n,l)$ such that the pair of vectors~\eqref{eq:pair}
is $1$-bad, that is the set of vectors for which $u(n+i)\ge u(l+i)$ for all $i=0,\ldots,r-1$,
or $u(n+i)\le u(l+i)$ for all $i=0,\ldots,r-1$. As the vectors are distinct, there exists an index $i=0,\ldots,r-1$ such that we have for example $u(n+i)> u(l+i)$, which means that $V_\bP(l)$ does not depend on the point $P_i$. The Bombieri bound given by Lemma~\ref{lem:Bomb} in the case when $f(X,Y)=X$ shows that for any fixed $c_1\in \E(\F_p)$ and $c_2\in\F_p$
$$
\sum_{P\in\E(\F_p)}\e_p(ax(c_1+P)+c_2) \ll 1+\sum_{P\in\E(\F_p), P\not=-c_1}\e_p(ax(c_1+P_i)+c_2)=O(p^{1/2}).
$$
So we bound our inner sum by summing over the point $P_i\in\E(\F_p)$ to obtain
\begin{equation*}
\begin{split}
\sum_{\bP\in \E(\F_p)^r} &
\e_p\(a
\(x(V_\bP(n)) - x(V_\bP(l))\)\)\\
=&\sum_{\bP_i\in \E(\F_p)^{r-1}}\sum_{P_i\in \E(\F_p)}
\e_p\(a
\(x(F_n(\bP_i)+P_i) - x(F_l(\bP_i))\)\)=O(p^{r-1/2}),
\end{split}
\end{equation*}
where $\bP_i$ is the vector obtained from $\bP$ by removing the point
$P_i$, and $F_m(P_i) \in \E(\F_p)$ denotes a point on $\E$ that depends
only on $m$ and $\bP_i$.

It remains to consider the case of $n,l\not\in\cB_r$. In this case,
there exist two indices $i,j=0,\ldots, r-1$ such that we
have for example $u(n+i)>u(l+i)$ and $u(n+j)<u(l+j)$.
Thus $V_\bP(l)$ does not depend on the point $P_i$ and
$V_\bP(n)$ does not depend on the point $P_j$. Using
Lemma~\ref{lem:Bomb} again, but this time applied for the sums over
the points $P_i$ and $P_j$ and~\eqref{eq:Hasse}, the inner sum becomes
\begin{eqnarray*}
\lefteqn{\sum_{\bP\in \E(\F_p)^r}
\e_p\(a
\(x(V_\bP(n)) - x(V_\bP(l))\)\)}\\
&=\sum_{\bP_{i,j}\in \E(\F_p)^{r-2}}\sum_{P_i\in \E(\F_p)}
\e_p\(a
\(x(G_n(\bP_{i,j})+P_i)\)\)\\
&\qquad\qquad  \sum_{P_j\in \E(\F_p)}
\e_p\(\(-ax(G_l(\bP_{i,j})+P_j)\)\)=O(p^{r-1}),
\end{eqnarray*}
where
$\bP_{i,j}$ is the vector obtained from $\bP$ by removing the points
$P_i$ and $P_j$, and $G_m(\bP_{i,j}) \in \E(\F_p)$ denotes a point on
$\E$ that depends only on $m$ and $\bP_{i,j}$.

Putting everything together, by Lemma~\ref{lem:Comp}, we obtain
\begin{eqnarray*}
\lefteqn{\( \sum_{\bP\in \E(\F_p)^r}
\left| \sum_{n=1}^{N_k}
\e_p\(ax(V_\bP(n))\)
\e_{N_k}\(\eta n\)\right|\)^2 }  \\
&&\qquad \ll p^{r}\(N_kp^r+p^{r-1/2}\sum_{\substack{n,l=1\\ n,l\in\cB_r}}^{N_k}1+p^{r-1}\sum_{\substack{n,l=1\\ n,l\not\in\cB_r}}^{N_k}1\)\\
&&\qquad  \ll N_kp^{2r}+3^r p^{2r-1/2}+ N_k^2p^{2r-1},
\end{eqnarray*}
and thus
\begin{eqnarray*}
\sum_{\bP\in \E(\F_p)^r}
\left| \sum_{n=1}^{N_k}
\e_p\(ax(V_\bP(n))\)
\e_{N_k}\(\eta n\)\right|  \ll N_k^{1/2}p^{r}+3^{r/2} p^{r-1/4}+ N_kp^{r-1/2}.
\end{eqnarray*}
Using~\eqref{eq:plogp}, we obtain
\begin{eqnarray*}
\sum_{\bP\in \E(\F_p)^r}\Delta_\bP(k)
&\ll& \(N_k^{1/2}p^{r}+3^{r/2} p^{r-1/4}+ N_kp^{r-1/2}\)\\
&& \qquad\qquad\qquad\sum_{0 < |a| < p} \frac{1}{|a|}
\sum_{\eta=0}^{N_k} \left|\sum_{\lambda=1}^N \e_{N_k}\(-\eta\lambda\)\right|\\
&\ll& \(N_k^{1/2}p^{r}+3^{r/2} p^{r-1/4}+ N_kp^{r-1/2}\) N_k\log N_k \log p\\
&\ll& \(N_k^{3/2}p^{r}+N_k 3^{r/2} p^{r-1/4}+ N_k^2p^{r-1/2}\)\log \tau \log p.
\end{eqnarray*}
Thus, for each $k=1,\ldots,\rf{\log \tau}$, the inequality
\begin{equation}
\label{eq:bad}
\Delta_\bP(k)\ge \delta^{-1}\(N_k^{3/2}+N_k 3^{r/2} p^{-1/4}+ N_k^2p^{-1/2}\)(\log \tau)^2 \log p
\end{equation}
holds for at most $O(\delta p^r/\log\tau)$ vectors $\bP \in \E(\F_p)^r$. Therefore, the number of vectors $\bP\in \E(\F_p)^r$ for which~\eqref{eq:bad} holds for at least one
$k=1,\ldots,\rf{\log \tau}$ is $O(\delta p^r)$. For all the other points
$\bP\in \E(\F_p)^r$, by~\eqref{DviaSigna} and taking into account that
$N_k=2N_{k-1}\le 2N$, we get
\begin{eqnarray*}
D_\bP(N)&\ll& \delta^{-1}N^{-1}\(N_k^{1/2}+3^{r/2} p^{-1/4}+ N_kp^{-1/2}\)(\log \tau)^2 \log p\\
&\ll&\delta^{-1}\(N^{-1/2}+3^{r/2}N^{-1} p^{-1/4}+ p^{-1/2}\)(\log \tau)^2 \log p,
\end{eqnarray*}
which concludes the proof.
\end{proof}

We note that El Mahassni~\cite{ElMa} obtained the bound
\begin{equation}
\label{eq:ElMa bound}
D_\bP(N) \ll \delta^{-1}\(N^{-1/2}+  p^{-1/4}\)(\log \tau)^2 \log p
\end{equation}
under the same conditions.
Say, in the most interesting case of sequences of maximal period
$\tau =2^r -1$ and $r$ chosen so that $2^r \ll p \ll  2^r$,
Theorem~\ref{th:D} gives a stronger result for
$$
\tau \ge N \ge \tau^{0.5\log 3/\log 2} = \tau^{0.79248 \ldots}.
$$

\subsection{Multidimensional Distribution}

For $\bP = (P_0, \ldots, P_{r-1}) \in \E(\F_p)^r$, we denote by $D_{\bP,s}(N)$ the $s$-dimensional discrepancy of the points~\eqref{eq:s-vec}.

\begin{theorem}
\label{th:Ds} Let the \lrs\ $\(u(n)\)_{n =1}^\infty$ be
purely periodic with period $\tau$ and order $r = O(p^{1/2})$ and let its
characteristic polynomial be irreducible over $\F_2$.
Then for any $\delta >0$, and for all except $O(\delta p^r)$ choices for $\bP \in \E(\F_p)^r$,
for all $1 \le N \le \tau$, we have
$$D_{\bP,s}(N) \ll  \delta^{-1}\(N^{-1/2}\log p+ p^{-1/2}\log p +\alpha_s^{r/2}N^{-1} (\log p)^s\)(\log \tau)^2,$$
where the implied constant depends only on $s$ and
 $\alpha_s$ is as in Lemma~\ref{lem:Comp}.
\end{theorem}

\begin{proof}
Exactly as in the proof of Theorem~\ref{th:D}, by Lemma~\ref{le:K-S} we get
\begin{equation}
\label{DviaSignaM}
D_{\bP,s}(N) \ll \frac{1}{p} + \frac{1}{N N_k} \Delta_{\bP,s}(k),
\end{equation}
where
\begin{eqnarray*}
\Delta_{\bP,s}(k) & = & \sum_{\substack{0 < |\va| < p\\ \va = (a_1,\ldots,a_s) \in \Z^s}} \frac{1}{r(\va)}
\sum_{\eta=0}^{N_k} \left|\sum_{\lambda=1}^N \e_{N_k}\(-\eta\lambda\)\right| \\
& &\qquad  \qquad \left| \sum_{n=1}^{N_k}
\e_p\(\sum_{\nu=0}^{s-1}a_\nu x(V_\bP(n+\nu))\) \e_{N_k}\(\eta n\)\right|.
\end{eqnarray*}

We further split the sum $\Delta_{\bP,s}(k)$ into two parts $\Delta_{\bP,s,1}(k)$
and $\Delta_{\bP,s,2}(k)$ where the summation in $\Delta_{\bP,s,1}(k)$ is
taken over the vectors $\va = (a_1,\ldots,a_s) \in \Z^s$ with only one non-zero
component and $\Delta_{\bP,s,2}(k)$ includes all other terms. Thus
\begin{equation}
\label{eq:split}
\Delta_{\bP,s}(k)= \Delta_{\bP,s,1}(k)+ \Delta_{\bP,s,2}(k).
\end{equation}

As in the proof of Theorem~\ref{th:D} we obtain
\begin{equation}
\label{eq:Delta1}
\sum_{\bP\in \E(\F_p)^r} \Delta_{\bP,s,1}(k)  \ll
\(N_k^{3/2}p^{r}+N_k 3^{r/2} p^{r-1/4}+ N_k^2p^{r-1/2}\)\log \tau \log p.
\end{equation}

Now, let $\cA_s$ be the set of the vectors $\va = (a_1,\ldots,a_s) \in
\Z^s$ with $0 < |\va| < p$ and with at least two nonzero
components. For $\va\in \cA_s$, applying the Cauchy inequality
 and Hasse bound~\eqref{eq:Hasse}, we
derive
\begin{eqnarray*}
\lefteqn{\( \sum_{\bP\in \E(\F_p)^r}
\left| \sum_{n=1}^{N_k}
\e_p\(\sum_{\nu=0}^{s-1}a_\nu x(V_\bP(n+\nu))\)
\e_{N_k}\(\eta n\)\right|\)^2 }  \\
& & \qquad \le  (p^{1/2}+1)^{2r}\sum_{\bP\in \E(\F_p)^r} \left| \sum_{n=1}^{N_k}
\e_p\(\sum_{\nu=0}^{s-1}a_\nu x(V_\bP(n+\nu))\) \e_{N_k}\(\eta n\)\right|^2 \\
& & \qquad \ll p^{r} \sum_{n,l=1}^{N_k} \e_{N_k}\(\eta (n-l)\)\\
& & \qquad \qquad  \qquad\sum_{\bP\in \E(\F_p)^r}
\e_p\(\sum_{\nu=0}^{s-1}a_\nu
\(x(V_\bP(n+\nu)) - x(V_\bP(l+\nu))\)\).
\end{eqnarray*}

Now, as in Theorem~\ref{th:D}, we split the sum into two sums,
one over pairs $(n,l)$ such that the pair of vectors~\eqref{eq:pair}
is $s$-bad and one over $s$-good pairs.

Let $\cB_{r,s}$ be the set of pairs of indices $(n,l)$ such that the pair of
vectors~\eqref{eq:pair} is $s$-bad.
For $(n,l)\in\cB_{r,s}$, as in the proof of Theorem~\ref{th:D},
we estimate the inner sum over $\bP$ trivially as $O(p^r)$.

It remains to consider the case of $(n,l)\not\in\cB_{r,s}$. Since $\va\in \cA_s$ there exist at least two distinct indices $i,j=0,\ldots,s-1$ such that $a_i,a_j\ne 0$. Since $(n,l)\not\in\cB_{r,s}$, there exist two indices $i_1,i_2=0,\ldots, r-s$ such that
\begin{equation*}
\begin{split}
&\(u(n+i_1), \ldots, u(n+i_1+s-1)\) = \ve_i , \\
&\(u(l+i_1), \ldots, u(l+i_1+s-1)\) = \mathbf{0}_s,
\end{split}
\end{equation*}
and
\begin{equation*}
\begin{split}
&\(u(n+i_2), \ldots, u(n+i_2+s-1)\) =  \mathbf{0}_s ,\\
&\(u(l+i_2), \ldots, u(l+i_2+s-1)\) =\ve_i .
\end{split}
\end{equation*}
Similarly, there exist two indices $j_1,j_2=0,1,\ldots,r-s$ such that
\begin{equation*}
\begin{split}
&\(u(n+j_1), \ldots, u(n+j_1+s-1)\) = \ve_j , \\
&\(u(l+j_1), \ldots, u(l+j_1+s-1)\) = \mathbf{0}_s,
\end{split}
\end{equation*}
and
\begin{equation*}
\begin{split}
&\(u(n+j_2), \ldots, u(n+j_2+s-1)\) =  \mathbf{0}_s ,\\
&\(u(l+j_2), \ldots, u(l+j_2+s-1)\) =\ve_j .
\end{split}
\end{equation*}

When $\nu\in\{1,2,\ldots,s\}\setminus\{i,j\}$, the equations above show that
$$
u(n+\nu+i_1)=u(n+\nu+i_2)=u(n+\nu+j_1)=u(n+\nu+j_2)=0,
$$
and so $V_\bP(n+\nu)$ does not depend on any of $P_{i_1},P_{i_2},P_{j_1},P_{j_2}$. Similarly, $V_\bP(l+\nu)$ does not depend on any of $P_{i_1},P_{i_2},P_{j_1},P_{j_2}$.

When $\nu=i$ (so $\nu\not=j$), the equations above show that
$$
u(n+\nu+i_1)=1\text{ and }u(n+\nu+i_2)=u(n+\nu+j_1)=u(n+\nu+j_2)=0,
$$
so $V_\bP(n+\nu)=V_\bP(n+i)$ depends on $P_{i_1}$, but does not depend on any of $P_{i_2},P_{j_1},P_{j_2}$. Similarly $V_\bP(l+i)$ depends on $P_{i_2}$, but none of $P_{i_1},P_{j_1},P_{j_2}$; $V_\bP(n+j)$ depends on $P_{j_1}$, but none of $P_{i_1},P_{i_2},P_{j_2}$; and $V_\bP(l+j)$ depends on $P_{j_2}$, but none of $P_{i_1},P_{i_2},P_{j_1}$.

Let $\bP_{i_1,i_2,j_1,j_2}$ be the vector obtained from $\bP$ after  discarding the points $P_{i_1}$, $P_{i_2}$, $P_{j_1}$ and $P_{j_2}$. We can apply Lemma~\ref{lem:Bomb} to our inner sum as in the one dimensional case, but this time applied for four sums over the points $P_{i_1}$, $P_{i_2}$, $P_{j_1}$ and $P_{j_2}$ to obtain
\begin{equation*}
\begin{split}\sum_{\bP\in \E(\F_p)^r}
\e_p&\(\sum_{\nu=0}^{s-1}a_\nu
\(x(V_\bP(n+\nu)) - x(V_\bP(l+\nu))\)\) \\
=&\sum_{\bP_{i_1,i_2,j_1,j_2}\in \E(\F_p)^{r-4}} \e_p\(\Psi_{n,l}(\bP_{i_1,i_2,j_1,j_2})\)\\
& \qquad \sum_{P_{i_1}\in \E(\F_p)}
\e_p\(a_i
\(x(G_n(\bP_{i_1,i_2,j_1,j_2})+P_{i_1})\)\)\\
&\qquad \quad  \sum_{P_{j_1}\in \E(\F_p)}
\e_p\(a_j
\(x(G_n(\bP_{i_1,i_2,j_1,j_2})+P_{j_1})\)\)\\
&\qquad \qquad \sum_{P_{i_2}\in \E(\F_p)}
\e_p\(-a_ix(G_l(\bP_{i_1,i_2,j_1,j_2})+P_{i_2})\)\\
& \qquad\qquad \quad  \sum_{P_{j_2}\in \E(\F_p)}
\e_p\(-a_jx(G_l(\bP_{i_1,i_2,j_1,j_2})+P_{j_2})\)\\
&=O\(p^{r-4}\(p^{1/2}\)^4 \) = O\(p^{r-2}\),
\end{split}
\end{equation*}
where
$$
\Psi_{n,l}(\bP_{i_1,i_2,j_1,j_2}) = \sum_{\substack{\nu=0\\\nu \ne i,j}}^{s-1}a_\nu
\(x(V_\bP(n+\nu)) - x(V_\bP(l+\nu))\)
$$
depends only on $n,l$ and $\bP_{i_1,i_2,j_1,j_2}$ and
$G_m(\bP_{i,j}) \in \E(\F_p)$ denotes a point on $\E(\F_p)$
that depends only on $m$ and $\bP_{i,j}$. (Note that we are using the fact that $a_i$ and $a_j$ are non-zero at this point.)

Putting everything together, by Lemma~\ref{lem:Comp}, we obtain
\begin{equation*}
\begin{split}
\( \sum_{\bP\in \E(\F_p)^r}
\left| \sum_{n=1}^{N_k}
\e_p\(\sum_{\nu=0}^{s-1}a_\nu x(V_\bP(n+\nu))\)
\e_{N_k}\(\eta n\)\right|\)^2  &\ll p^{r}\(N_k^2p^{r-2}+\alpha_s^r p^r\)\\
&\ll N_k^2p^{2r-2}+\alpha_s^r p^{2r},
\end{split}
\end{equation*}
and thus
\begin{eqnarray*}
\sum_{\bP\in \E(\F_p)^r}
\left| \sum_{n=1}^{N_k}
\e_p\(\sum_{\nu=0}^{s-1}a_\nu x(V_\bP(n+\nu))\)
\e_{N_k}\(\eta n\)\right|  \ll N_kp^{r-1}+ \alpha_s^{r/2} p^{r}.
\end{eqnarray*}
Now using~\eqref{eq:plogp}, we obtain
\begin{equation*}
\begin{split}
\sum_{\bP\in \E(\F_p)^r}\Delta_{\bP,s,2}&(k) \\
&\ll \(N_kp^{r-1}+ \alpha_s^{r/2} p^{r}\)
\sum_{\substack{0 < |\va| < p\\ \va = (a_1,\ldots,a_s) \in \Z^s}} \frac{1}{r(\va)}
\sum_{\eta=0}^{N_k} \left|\sum_{\lambda=1}^N \e_{N_k}\(-\eta\lambda\)\right|\\
&\ll \(N_kp^{r-1} + \alpha_s^{r/2} p^{r}\) N_k\log N_k (\log p)^s\\
&\ll \(N_k^2p^{r-1}+ \alpha_s^{r/2}N_k p^{r}\)\log \tau (\log p)^s.
\end{split}
\end{equation*}
Thus, from~\eqref{eq:split} and~\eqref{eq:Delta1},
we obtain the inequality
\begin{eqnarray*}
\sum_{\bP\in \E(\F_p)^r} \Delta_{\bP,s,}(k)
& \ll & \(N_k^{3/2}p^{r}+N_k 3^{r/2} p^{r-1/4}+ N_k^2p^{r-1/2}\)\log \tau \log p\\
& & \qquad \qquad\qquad\qquad + \(N_k^2p^{r-1}+ \alpha_s^{r/2}N_k p^{r}\)\log \tau (\log p)^s\\
& \ll & \(N_k^{3/2}p^{r}+  N_k^2p^{r-1/2}\)\log \tau \log p\\
 & & \qquad\qquad\qquad\qquad \qquad\qquad+~\alpha_s^{r/2}N_k p^{r}\log \tau (\log p)^s.
\end{eqnarray*}

Hence, for each $k=1,\ldots,\rf{\log \tau}$, we see that the inequality
\begin{equation}
\label{eq:bads}
\Delta_{\bP,s}(k)\ge \delta^{-1}\(N_k^{3/2}  \log p+  N_k^2p^{-1/2}  \log p+ \alpha_s^{r/2}N_k
(\log p)^s\)(\log \tau)^2
\end{equation}
holds for at most $O(\delta p^r/\log\tau)$ vectors $\bP \in \E(\F_p)^r$. Therefore, the number of vectors $\bP\in \E(\F_p)^r$ for which~\eqref{eq:bads} holds for at least
one $k=1,\ldots,\rf{\log \tau}$ is $O(\delta p^r)$. For all the other points $\bP\in \E(\F_p)^r$, by~\eqref{DviaSignaM} and taking into account that $N_k=2N_{k-1}\le 2N$, as in the proof of  Theorem~\ref{th:D} we get
\begin{equation*}
\begin{split}
D_{\bP,s}(N) \ll \delta^{-1}\(N^{-1/2}\log p+ p^{-1/2}\log p +\alpha_s^{r/2}N^{-1} (\log p)^s\)(\log \tau)^2,
\end{split}
\end{equation*}
which concludes the proof.
\end{proof}

Again, in the most interesting case of sequences of maximal period
$\tau =2^r -1$ and $r$ chosen so that $2^r \ll p \ll  2^r$,
Theorem~\ref{th:Ds} is nontrivial for
$$
\tau\ge N \ge \tau^{\gamma_s+\varepsilon} ,
$$
for any fixed $\varepsilon>0$ and a sufficiently  large $p$,
where
$$
\gamma_s = \frac{\log \alpha_s}{2 \log 2} < 1.
$$

\section{Comments}

We remark that the proofs of Theorems~\ref{th:D} and~\ref{th:Ds} depend
only on the fact that the binary vectors
$\(u(n+1),\ldots, u(n+r)\)$, $n =1, \ldots, \tau$, are pairwise
distinct. Thus the same results hold for many other sequences  $\(u(n)\)_{n =1}^\infty$,
for example for sequences generated by non-linear  recurrence relations.
In fact, in this generality, these results are new even in the case of the
classical
subset sum generator~\eqref{eq:SSG Zm} over a residue ring (as the proof in~\cite{ConShp}
applies only to \lrs{s}). Our method also applies to bounds of multiplicative character
sums with the sequence~\eqref{eq:SSG Zm} on average over
the vectors
$\vz = (z_0,\ldots\,, z_{r-1}) \in \Z_m^r$.

On the other hand, it is still  an open
problem to  obtain nontrivial results about the multidimensional distribution
of the elliptic curve subset
sum generator~\eqref{eq:SSG EC} on short segments. Note that the bound~\eqref{eq:ElMa bound}
is nontrivial starting from the values of $N$ of order $(\log \tau)^4 (\log p)^2$
(which can be further reduced  by using the approach of~\cite{Shp}).

\section*{Acknowledgements}

The second and the third authors are very grateful to the organisers of
the 2nd  International Conference on
Uniform Distribution Theory, 2010, Strobl (Austria) for the invitation
to this meeting, where the idea of this work was born.

During the preparation of this paper,  A.~O.\ was supported in part
by the Swiss National Science Foundation Grant~133399, I.~S.\ was
supported in part by the  Australian Research Council
Grant~DP0881473 and  by the National Research Foundation of
Singapore Grant~CRP2-2007-03.

\end{document}